\documentclass[11pt,leqno]{article}
\usepackage{amsmath, amscd, amsthm, amssymb, graphics, xypic, mathrsfs, setspace, fancyhdr, bm, pdfsync, enumitem, mathptmx}
\usepackage[usenames, dvipsnames, svgnames, table]{xcolor}
\usepackage[letterpaper,top=1.05in, bottom=1.05in, left=1.05in, right=1.05in]{geometry}
\usepackage[colorlinks=true,pagebackref=true]{hyperref}
\hypersetup{backref}

\newcommand{\colim}{\operatorname{colim}}

\newcommand{\Spec}{\operatorname{Spec}}

\newcommand{\isomto}{{\stackrel{\sim}{\;\longrightarrow\;}}}
\newcommand{\isomt}{{\stackrel{{\scriptscriptstyle{\sim}}}{\;\rightarrow\;}}}

\newcommand{\sma}{{\scriptstyle{\wedge}\,}}

\newcommand{\GW}{\mathbf{GW}}

\renewcommand{\hom}{\operatorname{Hom}}

\newcommand{\PP}{\mathbb{P}}

\newcommand{\cplx}{{\mathbb C}}

\newcommand{\Z}{{\mathbb Z}}
\newcommand{\N}{{\mathbb N}}
\newcommand{\A}{{\mathbb A}}

\newcommand{\aone}{{\mathbb A}^1}
\newcommand{\pone}{{\mathbb P}^1}

\newcommand{\gm}[1]{{{\mathbb G}_{m}^{#1}}}

\newcommand{\ho}[2][]{\mathrm{H}_{#1}({#2})}

\newcommand{\bpi}{\bm{\pi}}
\newcommand{\piaone}{{\bpi}^{\aone}}

\newcommand{\Nis}{{\operatorname{Nis}}}
\newcommand{\Zar}{\operatorname{Zar}} 

\newcommand{\aff}{\operatorname{aff}}

\newcommand{\Sm}{\mathrm{Sm}}

\newcommand{\Spc}{\mathrm{Spc}}

\newcommand{\K}{{{\mathbf K}}}

\newcommand{\KM}{\K^M}

\newcommand{\Singaone}{\operatorname{Sing}^{\aone}\!\!}
\newcommand{\op}[1]{\operatorname{#1}}

\renewcommand{\setminus}{\smallsetminus}

\newcommand{\Addresses}{{
		\bigskip
		\footnotesize
		
		A.~Asok, Department of Mathematics, University of Southern California, 3620 S.~Vermont Ave., Los Angeles, CA 90089-2532, United States; \textit{E-mail address:} \url{asok@usc.edu}
		
		\medskip
		
		J.~Fasel, Institut Fourier - UMR 5582, Universit\'e Grenoble Alpes, 100, rue des math\'ematiques, F-38402 Saint Martin d'H\`eres; France \textit{E-mail address:} \url{jean.fasel@univ-grenoble-alpes.fr}
}}

\newcounter{intro}
\theoremstyle{plain}
\newtheorem{thm}{Theorem}[subsection]

\newtheorem{lem}[thm]{Lemma}

\newtheorem*{claim*}{Claim} 

\newtheorem{question}[thm]{Question}

\newtheorem{conj}[thm]{Conjecture}
\newtheorem*{thm*}{Theorem}
\newtheorem*{problem*}{Problem}

\newtheorem{thmintro}{Theorem}

\newtheorem{conjintro}[thmintro]{Conjecture}

\theoremstyle{definition}
\newtheorem{defn}[thm]{Definition}

\theoremstyle{remark}
\newtheorem{rem}[thm]{Remark}

\newtheorem{ex}[thm]{Example}

\newtheorem{entry}[thm]{}

\numberwithin{equation}{subsection}

\begin{document}
	\pagestyle{fancy}
	\renewcommand{\sectionmark}[1]{\markright{\thesection\ #1}}
	\fancyhead{}
	\fancyhead[LO,R]{\bfseries\footnotesize\thepage}
	\fancyhead[LE]{\bfseries\footnotesize\rightmark}
	\fancyhead[RO]{\bfseries\footnotesize\rightmark}
	\chead[]{}
	\cfoot[]{}
	\setlength{\headheight}{1cm}

	\author{Aravind Asok\thanks{AA was partially supported by NSF Awards DMS-1802060  and DMS-2101898.} \and Jean Fasel}
	
	\title{{\bf Vector bundles on algebraic varieties}}
	\date{}
	\maketitle
	
	\begin{abstract}
		We survey some recent progress in the theory of vector bundles on algebraic varieties and related questions in algebraic K-theory.
	\end{abstract}


\section{Introduction}
The celebrated Poincar\'e--Hopf theorem implies that the vanishing locus of a suitably generic vector field on a closed, smooth manifold $M$ is topologically constrained: the number of points at which a generic vector field vanishes is equal to the Euler characteristic of $M$.  More generally, one may ask: given a vector bundle $E$ on a compact smooth manifold, what sorts of constraints are present on the topology of vanishing loci of generic sections?  If $M$ is a connected, closed, smooth manifold of dimension $d$ and $E$ is a rank $r$ vector bundle on $M$, then by the corank of $E$ we will mean the difference $d - r$.  The classical work of Eilenberg, Stiefel, Steenrod and Whitney laid down the foundations for results restricting the topology of vanishing loci of generic sections for bundles of a fixed corank; these results appear, essentially in modern form in Steenrod's book \cite{Steenrod}.  For example, one knows that if the corank of $E$ is negative, then $E$ admits a nowhere vanishing section and if the corank of $E$ is $0$, then a generic section vanishes at a finite set of points, and the cardinality of that finite set is determined by purely cohomological data (the Euler class of $E$ and the corresponding Euler number of $E$).  The situation becomes more interesting when the corank of $E$ is positive, to which we will return momentarily.  
 
In the mid 1950s, Serre created a dictionary between the theory of vector bundles in topology and the theory of projective modules over a commutative ring \cite{SerreFAC,SerreMP}. Echoing M.M. Postnikov's MathSciNet review of Serre's paper,
J.F. Adams prosaically wrote in his review of H. Bass' paper \cite{Bass}: ``This leads to the following programme: take definitions, constructions and theorems from bundle-theory; express them as particular cases of definitions, constructions and statements about finitely-generated projective modules over a general ring; and finally, try to prove the statements under suitable assumptions''.  One of the results Serre presented to illustrate this dictionary was the algebro-geometric analog of existence of nowhere vanishing sections for negative corank projective modules, now frequently referred to as Serre's splitting theorem, which we recall in algebro-geometric formulation: if $\mathcal{E}$ is a rank $r$ vector bundle over a Noetherian affine scheme $X$ of dimension $d$, then when $r > d$, $\mathcal{E} \cong \mathcal{E}' \oplus \mathcal{O}_{X}$.  

After the Pontryagin--Steenrod representability theorem, topological vector bundles on smooth manifolds (or spaces having the homotopy type of a CW complex) can be analyzed using homotopy theoretic techniques.  Extending Serre's analogy further and using celebrated work of Bass, Quillen, Suslin and Lindel, F. Morel showed that algebraic vector bundles on smooth affine varieties could be studied using an algebro-geometric homotopy theory: the Morel--Voevodsky motivic homotopy theory.    In this note, we survey recent developments in the theory of algebraic vector bundles motivated by this circle of ideas, making sure to indicate the striking analogies between topology and algebraic geometry.  

To give the reader a taste of the methods we will use, we mention two results here.   First, we state an improvement of Serre's splitting theorem mentioned above (for the moment it suffices to know that $\aone$-cohomological dimension is bounded above by Krull dimension, but can be strictly smaller).  Second, we will discuss the splitting problem for projective modules in corank $1$, which goes beyond any classical results.

\begin{thmintro}
	\label{thm:motivicSerre}
If $k$ is a field, and $X$ is a smooth affine $k$-scheme of $\aone$-cohomological dimension $\leq d$, then any rank $r > d$ bundle splits off a trivial rank $1$ summand.
\end{thmintro}

\begin{conjintro}
	\label{conj:murthy}
Assume $k$ is an algebraically closed field, and $X = \Spec R$ is a smooth affine $k$-variety of dimension $d$.  A rank $d-1$ vector bundle $\mathcal{E}$ on $X$ splits off a free rank $1$ summand if and only if $0 = c_{d-1}(\mathcal{E}) \in CH^{d-1}(X)$.
\end{conjintro}

In Theorem~\ref{thm:proofofmurthy} we verify Conjecture~\ref{conj:murthy} in case $d = 3,4$ (and $k$ has characteristic not equal to $2$).  To motivate the techniques used to establish these results we begin by analyzing topological variants of these conjectures.  We close this note with a discussion of joint work with Mike Hopkins which addresses the difficult problem of constructing interesting low rank vector bundles on ``simple" algebraic varieties.  As with any survey, this one reflects the biases and knowledge of the authors.  Limitations of space have prevented us from talking about a number of very exciting and closely related topics.

\section{A few topological stories}
\label{sec:topologicalstory}
In this section, we recall a few topological constructions that elucidate the approaches we use to analyze corresponding algebro-geometric questions studied later.  

\subsection{Moore--Postnikov factorizations}
\label{subsec:MoorePostnikov}
Suppose $f: E \to B$ is a morphism of pointed, connected topological spaces having the homotopy type of CW complexes that induces an isomorphism of fundamental groups (for simplicity of discussion).  Write $F$ for the ``homotopy'' fiber of $f$, so that there is a fiber sequence
\[
F \longrightarrow E \stackrel{f}{\longrightarrow} B
\] 
yielding a long exact sequence relating the homotopy of $F$, $E$ and $B$.  

A basic question that arises repeatedly is given a map $M \to B$, when can it be lifted along $f$ to a map $M \to E$? To approach this problem, one method is to factor $f$ in such a way as to break the original lifting problem into simpler problems where existence of a lift can be checked by, say, cohomological means.  

One systematic approach to analyzing this question was laid out in the work of Moore--Postnikov.  In this case, one factors $f$ so as to build $E$ out of $B$ by sequentially adding higher homotopy of $f$ (keeping track of the induced action of $\pi_1(E) \cong \pi_1(B)$ on the fiber).  In more detail, the Moore--Postnikov tower of $f$ consists of a sequence of spaces $\tau_{\leq i}f$, $i \geq 0$ and morphisms fitting into the following diagram:
\begin{equation}
		\label{eqn:MoorePostnikovdiagram}
\xymatrix{
	& & E \ar[d] \ar[dl]\ar[dr]& & \\
	\cdots \ar[r] & \tau_{\leq i+1} f \ar[r]\ar[dr]& \tau_{\leq i} f \ar[d] \ar[r]& \tau_{\leq i-1} f\ar[dl] \ar[r]& \cdots\\
	&&B. && 
}
\end{equation}
The key properties of this factorization are that i) the composite maps $E \to \tau_{\leq i}f \to B$ all coincide with $f$, ii) the maps $E \to \tau_{\leq i} f$ induce isomorphisms on homotopy groups in degrees $\leq i$, iii) the maps $\tau_{\leq i}f \to B$ induce isomorphisms on homotopy in degrees $> i+1$ and iv) there is a homotopy pullback diagram of the form
\begin{equation}
	\label{eqn:twistedprincipalfibration}
\xymatrix{
	\tau_{\leq i} f \ar[r]\ar[d] & \mathrm{B}\pi_1(E) \ar[d] \\
	\tau_{\leq i-1} f \ar[r] & K^{\pi_1(E)}(\pi_{i}(F),i+1).
}
\end{equation}
In particular, the morphism $\tau_{\leq i} f \to \tau_{\leq i-1}f$ is a twisted principal fibration, which means that a morphism $M \to \tau_{\leq i-1}f$ lifts along the tower if and only if the composite $M \to K^{\pi_1(E)}(\pi_{i}(F),i+1)$ lifts to $\mathrm{B}\pi_1(E)$.  The latter map amounts to a cohomology class on $M$ with coefficients in a local coefficient system; this cohomology class is pulled back from a ``universal example" the $k$-invariant at the corresponding stage.  If the obstruction vanishes, a lift exists.  Lifts are not unique in general but the ambiguity in choice of a lift can also be described.

\subsection{The topological splitting problem}
\label{subsec:corank}
In this section, to motivate some of the algebro-geometric results we will describe later, we review the problem of deciding whether a bundle of corank $0$ or $1$ on a closed smooth manifold $M$ of dimension $d+1$ has a nowhere vanishing section.  We now phrase this problem as a lifting problem of the type described in the preceding section.  

In this case, the relevant lifting problem is:
\[
\xymatrix{
	&  \mathrm{B}O(d-1) \ar[d]^{f} \\
	M \ar[r]_{\varphi}\ar@{-->}[ur]^{\exists ?} &  \mathrm{B}O(d).
}
\]
To analyze the lifting problem, we describe the Moore--Postnikov factorization of $f$.  The homotopy fiber of $f$ coincides with the standard sphere $S^{d-1} \cong O(d)/O(d-1)$.  

The stabilization map $O(d-1) \to O(d)$ is compatible with the determinant, and there are thus induced isomorphisms $\pi_1(\mathrm{B}O(d-1)) \to \pi_1(\mathrm{B}O(d)) \cong \Z/2$ compatible with $f$.  Note, however, that the action of $\Z/2$ on the higher homotopy of $\mathrm{B}O(d)$ depends on the parity of $d$: when $d$ is odd the action is trivial, while if $d$ is even the action is non-trivial in general and even fails to be nilpotent.  Of course $S^{d-1}$ is $(d-2)$-connected.  

\begin{rem}
At this stage, the fact that bundles of negative corank on spaces have the homotopy type of a CW complex of dimension $d$ follows immediately from obstruction theory granted the assertion that the sphere $S^r$ is an $(r-1)$-connected space in conjunction with the fact that negative corank means $r > d$.
\end{rem}

In order to write down obstructions, we need some information about the homotopy of spheres: the first non-vanishing homotopy group of $S^{d-1}$ is $\pi_{d-1}(S^{d-1})$ which coincides with $\Z$ for all $d \geq 2$ (via the degree map).  Likewise, $\pi_{d}(S^{d-1})$ is $\Z$ if $d = 3$ and $\Z/2$ if $d > 3$ and is generated by a suitable suspension of the classical Hopf map $\eta: S^3 \to S^2$.

Assume now $X$ is a space having the homotopy type of a finite CW complex of dimension $d+1$ for some fixed integer $d \geq 2$ (to eliminate some uninteresting cases) and $\xi: X \to \mathrm{B}O(d)$ classifies a rank $d$ vector bundle on $X$.  The first non-zero $k$-invariant for $f$ yields a map $X \to K^{\Z/2}(\pi_{d-1}(S^{d-1}),d)$, i.e., an element
\[
e(\xi) \in H^{d}(X,\Z[\sigma])
\]
called the (twisted) Euler class, where $\Z[\sigma]$ is $\Z$ twisted by the orientation character $\sigma$ defined by applying $\pi_1$ to the morphism $X \to \mathrm{B}O(d) \to \mathrm{B}(\Z/2)$.  

Assuming this primary obstruction vanishes, one may choose a lift to the next stage of the Postnikov tower.  If we fix a lift, then there is a well-defined secondary obstruction to lifting, that comes from the next $k$-invariant: this obstruction is given by a map $X \to K^{\Z/2}(\pi_d(S^{d-1}),d+1)$, i.e. a cohomology class in $H^{d+1}(X,\Z[\sigma])$ if $d=3$ or $H^{d+1}(X,\Z/2)$ if $d\neq 3$; in the latter case the choice of orientation character no longer affects this cohomology group.  

If one tracks the effect of choice of lift on the obstruction class described above, one obtains a map $K^{\Z/2}(\pi_{d-1}(S^{d-1}),d-1) \to K^{\Z/2}(\pi_d(S^{d-1}),d+1)$,
which is a twisted cohomology operation.  If $d = 3$, the map in question is a twisted version of the Pontryagin squaring operation, while if $d > 3$ the operation can be described as $Sq^2 + w_2 \cup$, where $w_2$ is the second Stiefel--Whitney class of the bundle.  In that case, the secondary obstruction yields a well-defined coset in 
\[
o_2(\xi) \in H^{d+1}(X,\Z/2)/(Sq^2 + w_2 \cup)H^{d-1}(X,\Z[\sigma])
\]
This description of the primary and secondary obstructions was laid out carefully by the early 1950s by S.D. Liao \cite{Liao}. 

Finally, the dimension assumption on $X$ guarantees that a lift of $\xi$ along $f$ exists if and only if these two obstructions vanish.  In principle, this kind of analysis can be continued, though the calculations become more involved as the indeterminacy created by successive choices of lifts becomes harder to control and information about higher unstable homotopy of spheres is also harder to obtain.  For a thorough treatment of this and even more general situations, we refer the reader to \cite{Thomas}.  

\begin{rem}
The analysis of the obstructions can be improved by organizing the calculations differently.  The Moore--Postnikov factorization has the effect of factoring a map $f: X \to Y$ as a tower of fibrations where the relevant fibers are Eilenberg--Mac Lane spaces.  However, there are many other ways to produce factorizations of $f$ with different constraints on the ``cohomological" properties of pieces of the tower.  
\end{rem}

\section{A quick review of motivic homotopy theory}
\label{sec:motivichomotopytheory}
Motivic homotopy theory, introduced by F. Morel and V. Voevodsky \cite{MV}, provides a homotopy theory for schemes over a base.  While there are a number of different approaches to constructing the motivic homotopy category that work in great generality, we work in a very concrete situation.  By an algebraic variety over a field $k$, we will mean a separated, finite type, reduced $k$-scheme.  We write $\Sm_k$ for the category of smooth algebraic varieties; for later use, we will also write $\Sm_k^{aff}$ for the full subcategory of $\Sm_k$ consisting of affine schemes.  

The category $\Sm_k$ is ``too small'' to do homotopy theory, in the sense that various natural categorical constructions one would like to make (increasing unions, quotients by subspaces, etc.) can leave the category.  As such, one first enlarges $\Sm_k$ to a suitable category $\Spc_k$ of ``spaces''; one may take $\Spc_k$ to be the category of simplicial presheaves on $\Sm_k$ and the functor $\Sm_k \to \Spc_k$ is given by the Yoneda embedding followed by the functor viewing a presheaf on $\Sm_k$ as a constant simplicial presheaf.  

\subsection{Homotopical sheaf theory}\label{subsec:homotopical}
Passing to $\Spc_k$ has the effect of destroying certain colimits that one would like to retain.  To recover the colimits that have been lost, one localizes $\Spc_k$ and passes to a suitable ``local'' homotopy category of the sort first studied in detail by K. Brown--S. Gersten, A. Joyal and J.F. Jardine: one fixes a Grothendieck topology $\tau$ on $\Sm_k$ and inverts the so-called $\tau$-local weak equivalences on $\Spc_k$; we refer the reader to \cite{JardineLocal} for a textbook treatment.  We write $\mathrm{H}_{\tau}(k)$ for the resulting localization of $\Spc_k$.  If $\mathscr{X} \in \Spc_k$, then a base-point for $\mathscr{X}$ is a morphism $x: \Spec k \to \mathscr{X}$ splitting the structure morphism.  There is an associated pointed homotopy category and these homotopy categories can be thought of as providing a convenient framework for ``non-abelian'' homological algebra.

Henceforth, we take $\tau$ to be the Nisnevich topology (which is finer than the Zariski topology, but coarser than the \'etale topology).  For the purposes of this note, it suffices to observe that the Nisnevich cohomological dimension of a $k$-scheme is equal to its Krull dimension, like the Zariski topology.

In the category of pointed spaces, we can make sense of wedge sums and smash products, just as in ordinary topology.  We also define spheres $S^i$, $i \geq 0$, as the constant simplicial presheaves corresponding to the simplicial sets $S^i$.  For any pointed space $(\mathscr{X},x)$, we define its homotopy sheaves $\bpi_i(\mathscr{X},x)$ as the Nisnevich sheaves associated with the presheaves on $\Sm_k$ defined by
\[
U \longmapsto \hom_{\mathrm{H}_{\Nis}(k)}(S^i \sma U_+,\mathscr{X},x);
\] 
here the subscript $+$ means adjoint a disjoint base-point.  These homotopy sheaves may be used to formulate a Whitehead theorem.


If $\mathbf{G}$ is a Nisnevich sheaf of groups on $\Sm_k$, then there is a classifying space $\op{B}\mathbf{G}$ such that for any smooth $k$-scheme $X$ one has a functorial identification of pointed sets of the form:
\[
\hom_{\mathrm{H}_{\Nis}(k)}(X,\op{B}\mathbf{G}) = \mathrm{H}^1_{\Nis}(X,\mathbf{G}).
\]
For later use, we set
\[
\mathrm{Vect}_n(X) := \mathrm{H}^1_{\Zar}(X,\op{GL}_n) = \mathrm{H}^1_{\Nis}(X,\op{GL}_n) = \hom_{\mathrm{H}_{\Nis}(k)}(X,\op{BGL}_n);
\]
where we as usual identify isomorphism classes of rank $n$ vector bundles locally trivial with respect to the Zariski topology on $X$ with $\op{GL}_n$-torsors (and the choice of topology does not matter).

If $\mathbf{A}$ is any Nisnevich sheaf of abelian groups on $\Sm_k$, then for any integer $n \geq 0$ there are Eilenberg--Mac Lane spaces $\mathrm{K}(\mathbf{A},n)$, i.e., spaces with exactly one non-vanishing homotopy sheaf, appearing in degree $n$, isomorphic to $\mathbf{A}$.  For such spaces, $\hom_{\mathrm{H}_{\Nis}(k)}(X,\mathrm{K}(\mathbf{A},n))$ has a natural abelian group structure, and there are functorial isomorphisms of abelian groups
\[
\hom_{\mathrm{H}_{\Nis}(k)}(X,\mathrm{K}(\mathbf{A},n)) = \mathrm{H}^n_{\Nis}(X,\mathbf{A}).
\]
With this definition, for essentially formal reasons there is a suspension isomorphism for Nisnevich cohomology with respect to the suspension $S^1 \sma (-)$. 


\subsection{The motivic homotopy category}
\label{subsec:A1story}
The motivic homotopy category is obtained as a further localization of $\mathrm{H}_{\Nis}(k)$: one localizes at the projection morphisms $\mathscr{X} \times \aone \to \mathscr{X}$.  We write $\mathrm{H}(k)$ for the resulting homotopy category; isomorphisms in this category will be referred to as $\aone$-weak equivalences.  Following the notation in classical homotopy theory, we write
\[
[\mathscr{X},\mathscr{Y}]_{\aone} := \hom_{\mathrm{H}(k)}(\mathscr{X},\mathscr{Y})
\] 
and refer to this set as the set of $\aone$-homotopy classes of maps from $\mathscr{X}$ to $\mathscr{Y}$.  

If $\mathscr{X}$ is a space, we will write $\bpi_0^{\aone}(\mathscr{X})$ for the Nisnevich sheaf associated with the presheaf $U \mapsto [U,\mathscr{X}]_{\aone}$ on $\Sm_k$; we refer to $\bpi_0^{\aone}(\mathscr{X})$ as the sheaf of connected components, and we say that $\mathscr{X}$ is $\aone$-connected if $\bpi_0^{\aone}(\mathscr{X})$ is the sheaf $\Spec(k)$.   

We consider $\gm{}$ as a pointed space, with base point its identity section $1$.  In that case, we define motivic spheres
\[
S^{i,j} := S^i \wedge \gm{\sma j}.
\]
We caution the reader that there are a number of different indexing conventions used for motivic spheres.  One defines bigraded homotopy sheaves $\bpi_{i,j}^{\aone}(\mathscr{X},x)$ for any pointed space as the Nisnevich sheaves associated with the presheaves on $\Sm_k$
\[
U \longmapsto [S^{i,j} \sma U_+, \mathscr{X}]_{\aone};
\]
we write $\bpi_i^{\aone}(\mathscr{X},x)$ for $\bpi_{i,0}^{\aone}(\mathscr{X})$.  We will say that a pointed space $(\mathscr{X},x)$ is $\aone$-$k$-connected for some integer $k \geq 1$ if it is $\aone$-connected and the sheaves $\bpi_i^{\aone}(\mathscr{X},x)$ are trivial for $1 \leq i \leq k$.  Because of the form of the Whitehead theorem in the Nisnevich local homotopy category, the sheaves $\bpi_{i}^{\aone}(-)$ detect $\aone$-weak equivalences.   

We write $\Delta^{\bullet}_k$ for the cosimplicial affine space with 
\[
\Delta^n_k := \Spec k[x_0,\ldots,x_n]/\langle \sum_i x_i = 1 \rangle.
\]  
For any space $\mathscr{X}$, we write $\Singaone \mathscr{X}$ for the space $\operatorname{diag} \underline{\hom}(\Delta^{\bullet},\mathscr{X})$.  There is a canonical map $\mathscr{X} \to \Singaone \mathscr{X}$ and the space $\Singaone \mathscr{X}$ is called the singular construction on $\mathscr{X}$.  For a smooth scheme $U$, the set of connected components $\pi_0(\Singaone \mathscr{X}(U))$ will be called the set of {\em naive $\aone$-homotopy classes} of maps $U \to \mathscr{X}$ (by construction, it is the quotient of the set of morphisms $U \to \mathscr{X}$ by the equivalence relation generated by maps $U \times \aone \to \mathscr{X}$).  Again, by definition there is a comparison morphism:
\begin{equation}
	\label{eqn:naivetruecomparison}
\pi_0(\Singaone \mathscr{X}(U)) \longrightarrow [U,\mathscr{X}]_{\aone}.
\end{equation}
Typically, the map \eqref{eqn:naivetruecomparison} is far from being a bijection.



\subsection{$\aone$-weak equivalences}\label{subsec:A1we}
We now give a number of examples of $\aone$-weak equivalences, highlighting some examples and constructions that will be important in the sequel.

\begin{ex}
	\label{ex:aonecontractiblesI}
A smooth $k$-scheme $X$ is called {\em ${\mathbb A}^1$-contractible} if the structure morphism $X \to \Spec k$ is an $\aone$-weak equivalence.  By construction, ${\mathbb A}^n$ is an $\aone$-contractible smooth $k$-scheme.  However, there are a plethora of $\aone$-contractible smooth $k$-schemes that are non-isomorphic to ${\mathbb A}^n$. For instance, the Russell cubic threefold, defined by the hypersurface equation $x + x^2y + z^2 + t^3 = 0$ is known to be non-isomorphic to affine space and also $\aone$-contractible \cite{DuboulozFasel}. See \cite{AsokOstvaer} for a survey of further examples.   
\end{ex}

\begin{ex}
If $f: X \to Y$ is a Nisnevich locally trivial morphism with fibers that are $\aone$-contractible smooth $k$-schemes, then $f$ is an $\aone$-weak equivalence.  Thus, the projection morphism for a vector bundle is an $\aone$-weak equivalence.  A vector bundle $E$ over a scheme $X$ can be seen as a commutative algebraic $X$-group scheme, so we may speak of $E$-torsors; $E$-torsors are classified by the coherent cohomology group $\mathrm{H}^1(X,\mathscr{E})$ (in particular, vector bundle torsors over affine schemes may always be trivialized).  Vector bundle torsors are Zariski locally trivial fiber bundles with fibers isomorphic to affine spaces, and the projection morphism for a vector bundle torsor is an $\aone$-weak equivalence.
\end{ex}

By an {\em affine vector bundle torsor} over a scheme $X$ we will mean a torsor $\pi: Y \to X$ for some vector bundle $E$ on $X$ such that $Y$ is an affine scheme.  Jouanolou proved \cite[Lemme 1.5]{Jouanolou73} that any quasi-projective variety admits an affine vector bundle torsor.  Thomason \cite[Proposition 4.4]{WeibelHKT} generalized Jouanolou's observation, and the following result is a special case of his results.

\begin{lem}[Jouanolou--Thomason homotopy lemma]
	\label{lem:JThomotopy}
	If $X$ is a smooth $k$-variety, then $X$ admits an affine vector bundle torsor.  In particular, any smooth $k$-variety is isomorphic in $\ho{k}$ to a smooth affine variety.  
\end{lem}

\begin{defn}
	\label{defn:Jdevice}
	By a {\em Jouanolou device} for a smooth $k$-variety $X$, we will mean a choice of an affine vector bundle torsor $p: Y \to X$. 
\end{defn}

\begin{ex}
	\label{ex:JdevicePn}
	When $X = {\mathbb P}^n$ there is a very simple construction of a ``standard'' Jouanolou device $\widetilde{\PP}^n$.  Geometrically, the standard Jouanolou device for ${\mathbb P}^n$ may be described as the complement of the incidence divisor in ${\mathbb P}^n \times {\mathbb P}^n$ where the second projective space is viewed as the dual of the first, with structure morphism the projection onto either factor.  
\end{ex}

\begin{ex}
	\label{ex:projectiveJdevice}
	If $X$ is a smooth projective variety of dimension $d$, then we may choose a finite morphism $\psi: X \to {\mathbb P}^d$.  Pulling back the standard Jouanolou device for ${\mathbb P}^d$ along $\psi$, we see that $X$ admits a Jouanolou device $\tilde{X}$ of dimension $2d$.  
\end{ex}

\begin{ex}
	\label{ex:oddquadric}
	For $n\in \N$, consider the smooth affine $k$-scheme $Q_{2n-1}$ defined as the hypersurface in ${\mathbb A}^{2n}_k$ given by the equation $\sum_{i=1}^n x_iy_i = 1$.  Projecting onto the first $n$-factors, we obtain a map $p:Q_{2n-1}\to \A^n\setminus 0$ which one may check is an affine vector bundle torsor.  For any integer $n \geq 0$, ${\mathbb A}^n \setminus 0$ is $\aone$-weakly equivalent to $S^{n-1,n}$ ( \cite[\S 3.2, Example 2.20]{MV}) and consequently $Q_{2n-1}$ is $\aone$-weakly equivalent to $S^{n-1,n}$ as well.
	
\end{ex}


\begin{ex}
	\label{ex:evenquadric}
For $n \in \N$, consider the smooth affine $k$-scheme $Q_{2n}$ defined as the hypersurface in ${\mathbb A}^{2n+1}_k$ given by the equation
\[
\sum_{i=1}^n x_i y_i = z(1-z).
\]
The variety $Q_{2}$ is isomorphic to the standard Jouanolou device over $\pone$.  The variety $\pone$ is $\aone$-weakly equivalent to $S^{1,1}$ and therefore $Q_2$ is $\aone$-weakly equivalent to $S^{1,1}$ as well. For $n \geq 2$, one knows that $Q_{2n}$ is $\aone$-weakly equivalent to $S^{n,n}$ \cite[Theorem 2]{AsokDoranFasel}.
\end{ex}

%
%



\subsection{Representability results}\label{subsec:representability}
If $\mathscr{F}$ is a presheaf on $\Sm_k$, we will say that $\mathscr{F}$ is $\aone$-invariant (resp. $\aone$-invariant on affines) if the pullback map $\mathscr{F}(X) \to \mathscr{F}(X \times \aone)$ is an isomorphism for all $X \in \Sm_k$ (resp. $X \in \Sm_k^{\aff}$).  A necessary condition for a cohomology theory on smooth schemes to be representable in $\mathrm{H}(k)$ is that it is $\aone$-invariant and has a Mayer--Vietoris property with respect to the Nisnevich topology.  One of the first functors that one encounters with these properties is that which assigns to a smooth $k$-scheme its Picard group.  Morel and Voevodsky showed \cite[\S 4 Proposition 3.8]{MV} that if $X$ is a smooth $k$-scheme, then the $\aone$-weak equivalence ${\mathbb P}^{\infty} \to \mathrm{B}\gm{}$ induces a bijection $[X,{\mathbb P}^{\infty}]_{\aone} \cong Pic(X)$.



If $\mathbf{A}$ is a sheaf of abelian groups on $\Sm_k$, then the functors $\mathrm{H}_{\Nis}^i(-,\mathbf{A})$ frequently fail to be $\aone$-invariant (taking $\mathbf{A} = \mathbb{G}_a$ gives a simple example) and therefore fail to be representable on $\Sm_k$.  The situation above where $\mathbf{A} = \gm{}$ provides the prototypical example of a sheaf whose cohomology {\em is} $\aone$-invariant (here the zeroth cohomology is the presheaf of units, which is even $\aone$-invariant on reduced schemes).  Following Morel and Voevodsky, we distinguish the cases where sheaf cohomology is $\aone$-invariant. 

\begin{defn}
	\label{defn:strongaoneinvariance}
A sheaf of groups $\mathbf{G}$ on $\Sm_k$ is called {\em strongly $\aone$-invariant} if for $i = 0,1$ the functors $\mathrm{H}_{\Nis}^i(-,\mathbf{G})$ on $\Sm_k$ are $\aone$-invariant.  A sheaf of abelian groups $\mathbf{A}$ on $\Sm_k$ is called {\em strictly $\aone$-invariant} if for all $i \geq 0$ the functors $\mathrm{H}_{\Nis}^i(-,\mathbf{A})$ on $\Sm_k$ are $\aone$-invariant.  
\end{defn}

The fundamental work of Morel, which we will review shortly, demonstrates the key role played by strongly and strictly $\aone$-invariant sheaves.  Nevertheless, various natural functors of geometric origin {\em fail} to be $\aone$-invariant on smooth schemes.  

\begin{ex}
	\label{ex:failureofaoneinvariance}
If $r \geq 2$, then the functor $\mathrm{H}_{\Nis}^1(-,GL_r)$ fails to be $\aone$-invariant on all schemes.  For an explicit example, consider the simplest case.  By a theorem of Dedekind--Weber frequently attributed to Grothendieck every rank $n$ vector bundle on $\pone$ is isomorphic to a unique line bundle of the form $\oplus_{i=1}^n \mathcal{O}(a_i)$ with the $a_i$ weakly increasing.  On the other hand, consider ${\pone} \times {\mathbb A}^1$ with coordinates $t$ and $x$.  The matrix
\[
\begin{pmatrix}
t \;\;& x \\
0 \;\;& t^{-1}
\end{pmatrix}
\]
determines a rank $2$ vector bundle on ${\pone} \times {\mathbb A}^1$ whose restriction to ${\pone} \times 0$ is $\mathcal{O}(1) \oplus \mathcal{O}(-1)$ and whose restriction to ${\pone} \times 1$ is $\mathcal{O} \oplus \mathcal{O}$.  
In contrast, Lindel's theorem affirming the Bass--Quillen conjecture in the geometric case shows that $\mathrm{H}_{\Nis}^1(-,GL_r)$ is $\aone$-invariant on affines.  The next result generalizes this last observation.
\end{ex}

\begin{thm}[Morel, Schlichting, Asok--Hoyois--Wendt]
	\label{thm:affinerepresentabilityVB}
	If $X$ is a smooth affine $k$-scheme, then for any $r \in \N$ there are functorial bijections of the form:
	\[
	\pi_0(\Singaone Gr_r(X)) \isomto [X,Gr_r]_{\aone} \isomto \mathrm{Vect}_r(X).
	\]	
\end{thm}

\begin{rem}
The above result was first established by F Morel in \cite{Morel08} for $r \neq 2$ and $k$ an infinite, perfect field, and his proof was partly simplified by M. Schlichting whose argument also established the case $r = 2$ \cite{SchlichtingEuler}.  The version above is stated in \cite{AHW}.
\end{rem}

\begin{rem}
While the functor of isomorphism classes of vector bundles is $\aone$-invariant on smooth affine $k$-schemes, even the latter can fail for $\op{G}$-torsors under more general group schemes, e.g., the special orthogonal group scheme $\op{SO}_n$ (\cite{Parimala78}, or \cite{Ojanguren78}).  Furthermore, while $\op{GL}_n$-torsors are always locally trivial with respect to the Nisnevich (and even the Zariski) topology, for an arbitrary smooth $k$-group scheme $\op{G}$, one only knows that $\op{G}$-torsors are locally trivial with respect to the \'etale topology.  

In \cite{AHWII,AHWIII}, it is shown that if $\op{G}$ is an isotropic reductive group scheme (see \cite[Definition 3.3.5]{AHWII} for the definition), then the functor assigning to $X \in \Sm_k^{aff}$ the set $\op{H}^1_{\Nis}(X,\op{G})$ is representable by $\op{BG}$.  This observation has a number of consequences, e.g., the following result about quadrics (see Examples~\ref{ex:oddquadric} and \ref{ex:evenquadric}).
\end{rem}

\begin{thm}[\cite{AHWII,AHWIII,AQuadric}]
	\label{thm:naivequadrics}
		For any integer $i \geq 1$, and any $X \in \Sm_k^{aff}$, the comparison map:
		\[
		\pi_0(\Singaone Q_{i}(X)) \isomto [X,Q_{i}]_{\aone}
		\]
		is a bijection, contravariantly functorial in $X$.
\end{thm}


%

\subsection{Postnikov towers, connectedness and strictly $\A^1$-invariant sheaves}
\label{subsec:strictlyA1invariant}
Recall from Definition~\ref{defn:strongaoneinvariance} the notion of strongly or strictly $\aone$-invariant sheaves of groups.  F. Morel showed that such sheaves can be thought of as ``building blocks" for the unstable $\aone$-homotopy category.  Morel's foundational works \cite{Morel05b, Morel08} can be viewed as a careful analysis of strictly and strongly $\aone$-invariant sheaves of groups and the relationship between the two notions.  More precisely, Morel showed that working over a perfect field $k$, the $\aone$-homotopy sheaves of a motivic space are {\em always} strongly $\aone$-invariant, and that strongly $\aone$-invariant sheaves of abelian groups are automatically strictly $\aone$-invariant.

To check this, Morel showed that strongly (resp. strictly) $\aone$-invariant sheaves of groups come equipped with a package of results/tools that are central to computations; this package of results is essentially an extension/amalgam/axiomatization of work of Bloch--Ogus and Gabber on \'etale cohomology exposed in \cite{CTHK} and Rost \cite{Rost96}.

\begin{ex}
	\label{ex:aoneinvariantsheaves}
Some examples of $\aone$-invariant sheaves that will appear in the sequel are:
\begin{itemize}
	\item unramified Milnor K-theory sheaves $\K^M_i$, $i \geq 0$ (see \cite[Corollary 6.5, Proposition 8.6]{Rost96} where, more generally, it is shown that any Rost cycle module gives rise to a strictly $\aone$-invariant sheaf);
	\item the Witt sheaf $\mathbf{W}$ or unramified powers of the fundamental ideal in the Witt ring $\mathbf{I}^j$, $j \geq 0$ (this follows from \cite{OjangurenPanin}); and 
	\item unramified Milnor--Witt K-theory sheaves $\K^{MW}_i$, $i \in \Z$ (see \cite[Chapter 3]{Morel08} for this assertion, or \cite[Corollary 8.5,Proposition 9.1]{Feld20a} where this observation is generalized to so-called Milnor--Witt cycle modules).
\end{itemize}
\end{ex}


\begin{entry}[Moore--Postnikov factorizations]
	\label{entry:aoneMoorePostnikov}
There is an analog of the Moore--Postnikov factorization of a map $f: \mathscr{E} \to \mathscr{B}$ of spaces along the lines described in Section~\ref{sec:topologicalstory}.   For concreteness we discuss the case where $\mathscr{E}$ and $\mathscr{B}$ are $\aone$-connected and $f$ induces an isomorphism on $\aone$-fundamental sheaves of groups for some choice of base-point in $\mathscr{E}$.  

Given $f$ as above, there are $\tau_{\leq i}f \in \Spc_k$ together with maps $\mathscr{E} \to \tau_{\leq i} f$, $\tau_{\leq i}f \to \mathscr{B}$ and $\tau_{\leq i}f \to \tau_{\leq i-1} f$ fitting into a diagram of exactly the same form as \eqref{eqn:MoorePostnikovdiagram} (replacing $E$ by $\mathscr{E}$ and $B$ by $\mathscr{B}$).  The relevant properties of this presentation are similar to those sketched before (replacing homotopy groups by homotopy sheaves), together with a homotopy pullback diagram of exactly the same form as \eqref{eqn:twistedprincipalfibration}.   We refer to this tower as the $\aone$--Moore--Postnikov tower of $f$ and the reader may consult \cite[Appendix B]{Morel08} or \cite[\S 6]{AFpi3a3minus0} for a more detailed presentation.  

If $X$ is a smooth scheme, then a map $\psi: X \to \mathscr{B}$ lifts to $\tilde{\psi}: X \to \mathscr{E}$ if and only if lifts exist at each stage of the tower, i.e. if and only if a suitable obstruction vanishes.  These obstructions are, by construction, valued in Nisnevich cohomology on $X$ with values in a strictly $\aone$-invariant sheaf (see \cite[\S 6]{AFpi3a3minus0} for a more detailed explanation).  
\end{entry}


By analogy with the situation in topology, we will use the $\aone$--Moore--Postnikov factorization to study lifting problems by means of obstruction theory. The relevant obstructions will lie in cohomology groups of a smooth scheme with coefficients in a strictly $\aone$-invariant sheaf. This motivates the following definition. 

\begin{defn}
	\label{defn:aonecohomologicaldimension}
Let $X$ be a smooth $k$-scheme. We say that $X$ has \emph{$\A^1$-cohomological dimension $\leq d$} if for any integer $i>d$ and any strictly $\A^1$-invariant sheaf $\mathbf{F}$, $\mathrm{H}_{\mathrm{Nis}}^i(X,\mathbf{F})=0$. In that case, we write $cd_{\A^1}(X)\leq d$.  
\end{defn}

\begin{ex}
	\label{ex:aonecohomologicaldimensionstrictbound}
	If $X$ is a smooth $k$-scheme of dimension $d$, then $X$ necessarily has $\aone$-cohomological dimension $\leq d$ as well.  Since ${\mathbb A}^n$ has $\aone$-cohomological dimension $\leq 0$, the $\aone$-cohomological dimension can be strictly smaller than Krull dimension; Example \ref{ex:projectiveJdevice} gives numerous other such examples. 
\end{ex} 
	

\subsection{Complex realization}\label{subsec:realization}
Assume $k$ is a field that admits an embedding $\iota_{{\mathbb C}}: k \hookrightarrow {\mathbb C}$. The functor that assigns to a smooth $k$-variety $X$ the complex manifold $X(\cplx)$ equipped with its classical topology extends to a complex realization functor 
\[
{\mathfrak R}_{{\cplx}}: \mathrm{H}(k) \longrightarrow \mathrm{H}
\]
where $\mathrm{H}$ is the usual homotopy category of topological spaces \cite[\S 3.3]{MV}. By construction, complex realization preserves finite products and homotopy colimits.  It follows that the complex realization of the motivic sphere $S^{p,q}$ is the ordinary sphere $S^{p+q}$ and consequently the complex realization functor induces group homomorphisms of the form:
\[
\bpi_{i,j}^{\aone}(X,x)(\cplx) \longrightarrow \pi_{i+j}(X(\cplx),x)
\]
for any pointed smooth $k$-scheme $(X,x)$.

Suppose $X$ is any $k$-scheme admitting a complex embedding and fix such an embedding. Write $\mathrm{Vect}_r^{top}(X)$ for the set of isomorphism classes of complex topological vector bundles on $X$.  There is a function
\[
\mathrm{Vect}_r(X) \longrightarrow \mathrm{Vect}_r^{top}(X)
\]
sending an algebraic vector bundle $E$ over $X$ to the topological vector bundle on $X(\cplx)$ attached to the base change of $E$ to $X_{\cplx}$.  We will say that an algebraic vector bundle is {\em algebraizable} if it lies in the image of this map. 

As rank $r$ topological vector bundles are classified by the set $[X(\cplx),\mathrm{B}U(r)]$ of homotopy classes of maps from $X(\cplx)$ to the complex Grassmannian, it follows that the function of the preceding paragraph factors as:
\[
\mathrm{Vect}_r(X) \longrightarrow [X,Gr_r]_{\aone} \longrightarrow \mathrm{Vect}_r^{top}(X).
\]
Theorem~\ref{thm:affinerepresentabilityVB} implies that the first map is a bijection if $X$ is a smooth affine $k$-scheme (or, alternatively, if $r = 1$).  More generally, combining Theorem~\ref{thm:affinerepresentabilityVB} and Lemma~\ref{lem:JThomotopy} one knows that any element of $[X,Gr_r]_{\aone}$ may be represented by an actual rank $r$ vector bundle on any Jouanolou device $\tilde{X}$ of $X$; this suggests the following definition.

\begin{defn}
	\label{defn:motivicvectorbundle}
If $X$ is a smooth $k$-scheme, then by a {\em rank $r$ motivic vector bundle} on $X$ we mean an element of the set $[X,Gr_r]_{\aone}$.
\end{defn}

\begin{question}
	\label{quest:algebraizability}
	If $X$ is a smooth complex algebraic variety, then which topological vector bundles are algebraizable (resp. motivic)?  
\end{question}


\section{Obstruction theory and vector bundles}
In order to apply the obstruction theory described in the previous sections to analyze algebraic vector bundles, we need more information about the structure of the classifying space $\op{BGL}_n$ including information about its $\aone$-homotopy sheaves, and the structure of the homotopy fiber of the stabilization map $\op{BGL}_n \to \op{BGL}_{n+1}$ induced by the map $\op{GL}_n \to \op{GL}_{n+1}$ sending an invertible matrix $X$ to the block matrix $diag(1,X)$. 

\subsection{The homotopy sheaves of the classifying space of $\op{BGL}_n$}
\label{subsec:motivicBorelHirzebruch}
We observed earlier that $\op{BGL}_1 = \mathrm{B}\gm{}$ is an Eilenberg--Mac Lane space for the sheaf $\gm{}$:  it is $\aone$-connected, and has exactly $1$ non-vanishing $\aone$-homotopy sheaf in degree $1$, which is isomorphic to $\gm{}$.  For $n \geq 1$, the analysis of homotopy sheaves of $\op{BGL}_n$ uses several ingredients.  First, Morel--Voevodsky observed that $\op{BGL} = \colim_n \mathrm{BGL}_n$ (for the inclusions described above) represents (reduced) algebraic K-theory after \cite[\S 4 Theorem 3.13]{MV}.  Second, Morel observed that there is an $\aone$-fiber sequence of the form
\begin{equation}
\label{eqn:stabilization}
{\mathbb A}^{n+1} \setminus 0 \longrightarrow \op{BGL}_n \longrightarrow \op{BGL}_{n+1},
\end{equation}
and that ${\mathbb A}^{n+1} \setminus 0$ is $\aone$-$(n-1)$-connected.  Furthermore, Morel computed \cite{Morel08} the first non-vanishing $\aone$-homotopy sheaf of ${\mathbb A}^{n+1} \setminus 0$ in terms of what he called Milnor--Witt K-theory sheaves (Example~\ref{ex:aoneinvariantsheaves}).  

Putting these ingredients together, one deduces:
\[
\bpi_i^{\aone}(\mathrm{BGL}_n) \cong \K^Q_i \quad\quad 1 \leq i \leq n-1,
\]
where $\K^Q_i$ is the (Nisnevich) sheafification of the Quillen K-theory presheaf on $\Sm_k$.  Following terminology from topology, sheaves in this range are called stable, and the case $i = n$ is called the first non-stable homotopy sheaf.  In \cite{Asok12b}, we described the first non-stable homotopy sheaf of $\op{BGL}_n$.  

The group scheme $\op{GL}_n$ maps to $\op{GL}_n(\cplx)$ under complex realization; the latter is homotopy equivalent to $U(n)$. For context, we recall some facts about homotopy of $U(n)$.  A classical result of Bott, refining results of Borel--Hirzeburch \cite[Theorem 25.8]{BorelHirzebruchII} asserts that the image of $\pi_{2n}(BU(n))$ in $H_{2n}(BU(n))$ is divisible by precisely $(n-1)!$ \cite{BottLoops}.  This result implies the assertion that $\pi_{2n}(U(n)) = n!$.  

Complex realization yields a map $\pi_{n,n}^{\aone}(\op{GL}_n) \longrightarrow \pi_{2n}(U(n))$.  One can view the celebrated ``Suslin matrices'' \cite{SuslinStablyFree} as providing an algebro-geometric realization of the generator of $\pi_{2n}(U(n))$.  Analyzing the fiber sequence of \eqref{eqn:stabilization} and putting all of the ingredients above together, we obtain the following result (we refer the reader to Example~\ref{ex:aoneinvariantsheaves} for notation).

   
\begin{thm}[{\cite[Theorem 1.1]{Asok12b}}]
	\label{thm:firstnonstableBGLn}
Assume $k$ is a field that has characteristic not equal to $2$.  For any integer $n \geq 2$, there are strictly $\aone$-invariant sheaves $\mathbf{S}_n$ fitting into exact sequences of the form:
\[
\begin{split}
0 \longrightarrow \mathbf{S}_{n+1} \longrightarrow &\piaone_n(\mathrm{BGL}_n) \longrightarrow \K^Q_n \longrightarrow 0 \;\;\;\;\; n \text{ odd }; \\
0 \longrightarrow \mathbf{S}_{n+1}\times_{\KM_{n+1}/2}\mathbf{I}^{n+1}\to &\piaone_n(\mathrm{BGL}_n)\longrightarrow \K^Q_n \longrightarrow 0 \;\;\;\;\; n \text { even, }
\end{split}
\]
where 
\begin{itemize}
	\item[(1)] there is a canonical epimorphism $\KM_n/(n-1)!\to \mathbf{S}_n$ which becomes an isomorphism after $n-2$ contractions (see \cite[\S 2.3]{Asok12b} for this terminology);
	\item[(2)] there is a canonical epimorphism $\mathbf{S}_n\to \KM_n/2$ such that the composite \[
	\KM_n/(n-1)!\to \mathbf{S}_n\to  \KM_n/2
	\] is reduction modulo $2$.
	\item[(3)] the fiber product is taken over the epimorphism $\mathbf{S}_{n+1}\to \KM_{n+1}/2$ and a sheafified version of Milnor's homomorphism $\mathbf{I}^{n+1}\to \KM_{n+1}/2$.
\end{itemize}
Moreover, if $k$ admits a complex embedding, then the map $\bpi_{n,n+1}^{\aone}(\op{BGL}_n)(\cplx) \longrightarrow \pi_{2n+1}(BU(n)) \cong \Z/n!$ induced by complex realization is an isomorphism.
\end{thm}


Bott's refinement of the theorem of Borel--Hirzebruch turns out to have an algebro-geometric interpretation.  Indeed, in joint work with T.B. Williams \cite{AFWSuslinConj} we showed that $\mathbf{S}_n$ can described using a ``Hurewicz map'' analyzed by Andrei Suslin \cite{SuslinHurewicz}.  Suslin's conjecture on the image of this map is equivalent to the following conjecture.

\begin{conj}[Suslin's factorial conjecture]
	The canonical epimorphism $\KM_n/(n-1)!\to \mathbf{S}_n$ is an isomorphism.
\end{conj}

\begin{rem}
The conjecture holds tautologically for $n=2$.  For $n=3$, Suslin observed the conjecture was equivalent to the Milnor conjecture on quadratic forms, which was resolved later independently by Merkurjev--Suslin and Rost.  The conjecture was established for $n=5$ in ``most'' cases in \cite{AFWSuslinConj} (see the latter for a precise statement); this work relies heavily on the computation by \O stv\ae r-R\"ondigs-Spitzweck of the motivic stable $1$-stem \cite{RSO}.
\end{rem}

\subsection{Splitting bundles, Euler classes and cohomotopy}
Morel's computations around ${\mathbb A}^n \setminus 0$ in conjunction with the fiber sequences of \eqref{eqn:stabilization} allow a significant improvement of Serre's celebrated splitting theorem for smooth affine varieties over a field that we stated in the introduction.

\begin{proof}[Proof of the motivic Serre Splitting Theorem~\ref{thm:motivicSerre}]
Suppose $X$ is a smooth affine $k$-variety having $\aone$-cohomological dimension $\leq d$, and suppose $\xi: X \to \op{BGL}_r$ classifies a rank $r > d$ vector bundle on $X$.  We proceed by analyzing the $\aone$--Moore--Postnikov factorzation of the stabilization map  \eqref{eqn:stabilization} with $n = {r-1}$.  In that case, combining the fact that ${\mathbb A}^{r} \setminus 0$ is $\aone$-$(r-2)$-connected and the $\aone$-cohomological dimension assumption on $X$ one sees all obstructions to splitting vanish.
\end{proof}


\begin{rem}
	The proof of this result does not rely on the Serre splitting theorem.  Since $\aone$-cohomological dimension can be strictly smaller than Krull dimension (Example~\ref{ex:aonecohomologicaldimensionstrictbound}),  this statement is strictly stronger than Serre splitting.  Importantly, the improvement achieved here seems inaccessible to classical techniques.
\end{rem}

The algebro-geometric splitting problem in corank $0$ on smooth affine varieties of dimension $d$ over a field $k$ has been analyzed by many authors.  When $k$ is an algebraically closed field, M.P. Murthy proved that the top Chern class in Chow groups is the only obstruction to splitting \cite{Murthy94}.  When $k$ is not algebraically closed, vanishing of the top Chern class is known to be insufficient to guarantee splitting, and Nori proposed some ideas to analyze this situation.  His ideas led Bhatwadekar and Sridharan \cite{BhatwadekarSridharan} to introduce what they called Euler class groups and to provide one explicit ``generators and relations'' answer to this question.  At the same time, F. Morel proposed an approach to the splitting problem in corank $0$, which we recall here.

\begin{thm}[Morel's splitting theorem {\cite[Theorem 1.32]{Morel08}}]
	\label{thm:morelsplitting}
Assume $k$ is a field and $X$ is a smooth affine $k$-variety of $\aone$-cohomological dimension $\leq d$.  If $\mathcal{E}$ is a rank $d$ vector bundle on $X$, then $\mathcal{E}$ splits off a free rank $1$ summand if and only if an Euler class $e(\mathcal{E}) \in H^d(X,\K^{MW}_d(\det \mathcal{E}))$ vanishes.
\end{thm}

 
\begin{rem}
The Euler class of Theorem~\ref{thm:morelsplitting} is precisely the first non-vanishing obstruction class, as described in Paragraph~\ref{entry:aoneMoorePostnikov}.
A related ``cohomological" approach to the splitting problem in corank 0 was proposed by Barge--Morel \cite{BargeMorel} and analyzed in the thesis of the second author \cite{FaselChowWitt}.  The cohomological approach was in most cases shown to be equivalent to the ``obstruction theoretic'' approach in \cite{AsokFaselComparison}. 
\end{rem}

The following result shows that the relationship between Euler classes \`a la Bhatwadekar--Sridharan and Euler classes \`a la Morel is mediated by another topologically inspired notion: cohomotopy (at least for bundles of trivial determinant).

\begin{thm}[{\cite[Theorem 1]{AsokFaselDas}}]
Suppose $k$ is a field, $n$ and $d$ are integers, $n \geq 2$, and $X$ is a smooth affine $k$-scheme of dimension $d \leq 2n-2$.  Write $\mathrm{E}^n(X)$ for the Bhatwadekar--Sridharan Euler class group.
\begin{itemize}
	\item The set $[X,Q_{2n}]_{\aone}$ carries a functorial abelian group structure;
	\item There are functorial homomorphisms:
	\[
	\mathrm{E}^n(X) \stackrel{s}{\longrightarrow} [X,Q_{2n}]_{\aone} \stackrel{h}{\longrightarrow } \op{H}^n(X,\K^{MW}_n)
	\]
	where the ``Segre class'' homomorphism $s$ is surjective and an isomorphism if $k$ is infinite and $d \geq 2$, and the Hurewicz homomorphism $h$ is an isomorphism if $d \leq n$.
\end{itemize}	
\end{thm}

\begin{rem}
The group structure on $[X,Q_{2n}]_{\aone}$ is an algebro-geometric variant of Borsuk's group structure on cohomotopy.  The second point of the statement includes the algebro-geometric analog of the Hopf classification theorem from topology.
\end{rem}

\subsection{The next nontrivial $\aone$-homotopy sheaf of spheres}
In Section~\ref{subsec:corank} we described a cohomological approach to the splitting problem in corank $1$ for smooth closed manifolds of dimension $d$; this approach relied on the computation of $\pi_{d}(S^{d-1})$.  In order to analyze the algebro-geometric splitting problem in corank $1$ using the $\aone$--Moore--Postnikov factorization we will need as input further information about the homotopy sheaves of ${\mathbb A}^d \setminus 0$.  We now describe known results in this direction.  For technical reasons, we assume $2$ is invertible in what follows.

\subsubsection{The $KO$-degree map}\label{subsec:KOdegree}
In classical algebraic topology, all of the ``low degree'' elements in the homotopy of spheres can be realized by constructions of ``linear algebraic'' nature.  The situation in algebraic geometry appears to be broadly similar.  The first contribution to the ``next'' non-trivial homotopy sheaves of motivic spheres requires recalling the geometric formulation of Bott periodicity for Hermitian K-theory given by Schlichting--Tripathi.  

We write $\mathrm{O}$ for the infinite orthogonal group.  In topology, Bott periodicity identifies the $8$-fold loop space of $\mathrm{O}$ with itself and identifies the intermediate loop spaces in concrete geometric terms.  In algebraic geometry, Schlichting and Tripathi proved that the $4$-fold $\pone$-loop space $\Omega_{\pone}^4 \mathrm{O}$ also coincides with $\mathrm{O}$ and realized suitable intermediate loop spaces: $\Omega_{\PP^1}^n \mathrm{O}$ is isomorphic to $\op{GL}/\op{O}$ when $n = 1$, $\op{Sp}$ when $n = 2$ and $\op{GL}/\op{Sp}$ when $n = 3$, where $\mathrm{Sp}$ is the stable sympletic group, $\mathrm{GL}/\mathrm{O}$ is the ind-variety of invertible symmetric matrices, and  $\mathrm{GL}/\mathrm{Sp}$ is the ind-variety of invertible skew-symmetric matrices {\cite[Theorems 8.2 and 8.4]{Schlichting13}}.

A slight modification of the Suslin matrix construction \cite[Lemma 5.3]{SuslinStablyFree} yields a map
\[
u_n:Q_{2n-1} \longrightarrow \Omega^{-n}_{\pone}O
\]
called the \emph{(unstable) $\mathrm{KO}$-degree map in weight $n$} that was analyzed in detail in \cite{AFKODegree}.  The terminology stems from the fact that this map stabilizes to the ``unit map from the sphere spectrum to the Hermitian K-theory spectrum'' in an appropriate sense.  The scheme $Q_{2n-1}$ is $\A^1$-$(n-2)$-connected by combining the weak equivalence of Example~\ref{ex:oddquadric} and Morel's connectivity results for ${\mathbb A}^n \setminus 0$.  Thus, $u_n$ factors through the $\aone$-$(n-2)$-connected cover of $\Omega^{-n}_{\pone}O$. 

Taking homotopy sheaves on both sides, there are induced morphisms: 
\[
\piaone_{i}(u):\piaone_{i}(Q_{2n-1})\to \piaone_{i}(\Omega^{-n}_{\pone}O).
\]
This homomorphism is trivial if $i < n-1$ by connectivity estimates.  If $i = n-1$, via Morel's calculations one obtains a morphism $\K^{\mathrm{MW}}_n\to \GW^n_n$ whose sections over finitely generated field extensions of $k$ can be viewed as a quadratic enhancement of the ``natural'' map from Milnor $K$-theory to Quillen $K$-theory defined by symbols; we will refer to it as the natural homomorphism (the natural homomorphism is known to be an isomorphism if $n \leq 4$; the case $n \leq 2$ is essentially Suslin's, $n = 3$ is \cite[Theorem 4.3.1]{AFKODegree}, and $n = 4$ is unpublished work of O. R\"ondigs). 


When $i=n$, we obtain a morphism:
\[
\piaone_n({\mathbb A}^n \setminus 0) \cong \piaone_n(Q_{2n-1})\to \piaone_n(\Omega^{-n}_{\pone}O) \cong \GW_{n+1}^n,
\]
where the right-hand term is by definition a higher Grothendieck--Witt sheaf (obtained by sheafifying the corresponding higher Grothendieck-Witt presheaf on $\Sm_k$).  The above map is an epimorphism for $n = 2,3$ and it follows from these observations that the morphism is an epimorphism after $(n-3)$ contractions (\cite[Theorem 4.4.5]{AFKODegree}). 

\subsubsection{The motivic J-homomorphism}
The classical J-homomorphism has an algebro-geometric counterpart that yields the second contribution to the ``next'' homotopy sheaf of motivic spheres.  The standard action of $\op{SL}_n$ on ${\mathbb A}^n$ extends to an action on the one-point compactification ${\mathbb P}^n/{\mathbb P}^{n-1}$.  The latter space is a motivic sphere ${\pone}^{\sma n}$ and thus one obtains a map
\[
\Sigma_{{\pone}}^n SL_n \longrightarrow {\pone}^{\sma n}. 
\]
As $SL_n$ is $\aone$-connected, it follows that $\Sigma^n_{\pone} SL_n$ is $\A^1-n$-connected.

The first non-vanishing $\aone$-homotopy sheaf appears in degree $n+1$; for $n = 2$, it is isomorphic to $\K^{MW}_4$, while for $n \geq 3$ it is isomorphic to $\K^M_{n+2}$; this follows from $\aone$-Hurewicz theorem combined with \cite[Proposition 3.3.9]{AWW} using the fact that $\bpi_1^{\aone}(SL_n) = \K^M_2$ for $n \geq 3$ and properties of the $\aone$-tensor product \cite[Lemma 5.1.8]{AWW}.


Combining the above discussion with that of the previous section, we see that for $n \geq 3$, we may consider the composite maps $\K^M_{n+2} \to \bpi_{n+1}^{\aone}({\pone}^{\sma n}) \to \mathbf{GW}^n_{n+1}$; this composite is known to be zero, but the map induced by the J-homomorphism fails to be injective.  Instead, it factors through a morphism
\[
\K^M_{n+2}/24 \longrightarrow \bpi_{n+1}^{\aone}({\pone}^{\sma n}) \longrightarrow \mathbf{GW}^n_{n+1}.
\]
Furthermore, the map on the right fails to be surjective.  The unstable description above is not present in the literature, but it is equivalent to the results stated in \cite{AFWSuslinConj}.  In \cite{RSO}, the stable motivic $1$-stem was computed in the terms above: the above sequence is exact on the left stably.  The next result compares the unstable group to the corresponding stable group.




\begin{thm}
	\label{thm:pinanminus0}
For any integer $n \geq 3$, the kernel $\mathbf{U}_{n+1}$ of the stabilization map
\[
\bpi_{n+1}^{\aone}({\pone}^{\sma n}) \longrightarrow \bpi_{n+1}^{\aone}(\Omega^{\infty}_{\pone} \Sigma^{\infty}_{\pone} {\pone}^{\sma n})
\]
is a direct summand; the stabilization map is an isomorphism if $n = 3$, i.e. $\mathbf{U}_4 = 0$.
\end{thm}

\begin{conj}\label{conj:unzero}
For $n \geq 4$, the sheaf $\mathbf{U}_{n+1}$ is zero. 
\end{conj}

\begin{rem}
Conjecture~\ref{conj:unzero} would follow from a suitable version of the Freudenthal suspension theorem for $\pone$-suspension.
\end{rem}


\subsection{Splitting in corank $1$}
Using the results above, we can analyze the splitting problem for vector bundles in corank $1$. The expected result was posed as a question by Murthy (\cite[p. 173]{Murthy99}) which we stated in the introduction as Conjecture~\ref{conj:murthy}.  Murthy's conjecture is trivial if $d = 2$.  In \cite{AFBundle} and \cite{AFpi3a3minus0} we established the following result, which reduces Murthy's question to Conjecture~\ref{conj:unzero}.
 

\begin{thm}
	\label{thm:proofofmurthy}
	Let $X$ be a smooth affine scheme of dimension $d \geq 2$ over an algebraically closed field $k$.  A rank $d-1$ vector bundle $\mathcal{E}$ on $X$ splits off a trivial rank $1$ summand if and only if $c_{d-1}(\mathcal{E}) \in \mathrm{CH}^{d-1}(X)$ is trivial and a secondary obstruction $o_2(\mathcal{E}) \in H^d(X,\bpi_{d-1}^{\aone}({\mathbb A}^{d-1} \setminus 0))$ vanishes.  This secondary obstruction vanishes if $d = 3,4$ or if Conjecture~\ref{conj:unzero} has a positive answer.
\end{thm}

To establish this result, one uses the assumptions that $X$ is smooth affine of Krull dimension $d$ and $k$ is algebraically closed in a strong way.  Indeed, these assertions can be leveraged to show that the primary obstruction, which is a priori an Euler class, actually coincides with the $(d-1)$st Chern class.  The secondary obstruction can be described by Theorem~\ref{thm:pinanminus0} and the form of the secondary obstruction is extremely similar to Liao's description in Section~\ref{subsec:corank}: it is a coset in $\mathrm{Ch}^{d}(X)/(Sq^2 + c_1(\mathcal{E}) \cup)\mathrm{Ch}^{d-1}(X)$ where $\mathrm{Ch}^i(X) = \mathrm{CH}^i(X)/2$.  Once more, the assumptions on $X$ guarantee that $\mathrm{Ch}^d(X)$ is trivial and thus the secondary obstruction is so as well.

\subsection{The enumeration problem}
If a vector bundle $E$ splits off a free rank $1$ summand, then another natural question is to enumerate the possible $E'$ that become isomorphic to $E$ after adding a free rank $1$ summand.  This problem may also be analyzed in homotopy theoretic terms as it amounts to enumerating the number of distinct lifts.  This kind of problem was studied in detail in topology by James and Thomas \cite{James65} and the same kind of analysis can be pursued in algebraic geometry.  

The history of the enumeration problem in algebraic geometry goes back to early days of algebraic K-theory.  Indeed, the Bass--Schanuel cancellation theorem \cite{Bass} solves the enumeration problem for bundles of negative corank.  Suslin's celebrated cancellation theorem \cite{Suslin77} solved the enumeration problem in corank $0$.  In all of these statements, ``cancellation" means that there is a unique lift.  On the other hand, Mohan Kumar observed \cite{MohanKumarstablyfree} that for bundles of corank $2$, uniqueness was no longer true in general.  Nevertheless, Suslin conjectured that the enumeration problem had a particularly nice solution in corank $1$.

\begin{conj}[Suslin's cancellation conjecture]
	If $k$ is an algebraically closed field, and $X$ is a smooth affine $k$-scheme of dimension $d\geq 2$.  If $\mathcal{E}$ and $\mathcal{E}'$ are corank $1$ bundles that become isomorphic after addition of a trivial rank $1$ summand, then $\mathcal{E}$ and $\mathcal{E}'$ are isomorphic.
\end{conj}

The above conjecture is trivial when $d = 2$.  It was established for $\mathcal{E}$ the trivial bundle of rank $d-1$ in \cite{Fasel10b} and $d=\mathrm{dim}(X)$ under the condition that $(d-1)!$ is invertible in $k$.  The above conjecture was also established for $d = 3$ in \cite{AFBundle} (assuming $2$ is invertible in $k$).  Paralleling the results of James-Thomas in topology (\cite{James65}), P. Du was able to prove in \cite{Du20} that Suslin's question has a positive answer for oriented vector bundles in case the cohomology group $\mathrm{H}^d_{\Nis}(X,\piaone_{d}(\A^{d}\setminus 0))$ vanishes. This vanishing statement would follow immediately from Conjecture~\ref{conj:unzero}.


\section{Vector bundles: non-affine varieties and algebraizability}
In this final section, we survey some joint work with M. J. Hopkins related to the classification of motivic vector bundles (see Definition~\ref{defn:motivicvectorbundle}), its relationship to the algebraizability question (see Question~\ref{quest:algebraizability}), and investigate the extent to which $\aone$-homotopy theory can be used to analyze vector bundles on projective varieties.

\subsection{Descent along a Jouanolou device}
If $X$ is a smooth algebraic $k$-variety, then there is always the map
\begin{equation}
	\label{eqn:Jdescent}
\mathrm{Vect}_r(X) \longrightarrow [X,\op{BGL}_r]_{\aone}
\end{equation}
from rank $r$ vector bundles to rank $r$ motivic vector bundles.  When $X$ is affine, Theorem~\ref{thm:affinerepresentabilityVB} guarantees that this map is a bijection, and examples show that the map fails to be an isomorphism outside of this case.  Nevertheless, it is very interesting to try to quantify the failure of the above map to be a bijection.

If $\pi: \tilde{X} \to X$ is a Jouanolou device for $X$, then it follows from the definitions that the map \eqref{eqn:Jdescent} coincides with $\pi^*: \mathrm{Vect}_r(X) \to \mathrm{Vect}_r(\tilde{X})$ under the bijection of Theorem~\ref{thm:affinerepresentabilityVB}.  The morphism $\pi$ is faithfully flat by construction, and therefore, vector bundles on $X$ are precisely vector bundles on $\tilde{X}$ equipped with a descent datum along $\pi$.  

Since $\pi: \tilde{X} \to X$ is an affine morphism, it follows that $\tilde{X} \times_X \tilde{X}$ is itself an affine scheme, and the two projections $p_1,p_2: \tilde{X} \times_X \tilde{X} \to \tilde{X}$ are $\aone$-weak equivalences.  Thus, pullbacks $p_1^*$ and $p_2^*$ are bijections on sets of isomorphism classes of vector bundles.  In fact, since the relative diagonal map splits the two projections, the two pullbacks actually coincide on isomorphism classes.  In descent-theoretic terms, these observations mean that any vector bundle $\mathscr{E}$ on $\tilde{X}$ can always be equipped with an isomorphism $p_1^* \mathscr{E} \isomt p_2^* \mathscr{E}$, i.e., a pre-descent datum.  Thus, the only obstruction to descending a vector bundle along $\pi$ is whether one may choose a pre-descent datum that actually satisfies the cocycle condition.  With this observation in mind, it seems natural to analyze the question of whether {\em every} vector bundle can be equipped with a descent datum along $\pi$.

\begin{question}\label{quest:Jdescent}
	If $X$ is a smooth $k$-variety and $\pi:\tilde{X}\to X$ is a Jouanolou device for $X$, then is the pull-back map 
	\[
	p^*:\mathrm{Vect}_n(X)\to \mathrm{Vect}_n(\tilde{X})
	\]
	surjective?
\end{question}


\begin{thm}[Asok, Fasel, Hopkins]
	Suppose $X$ is a smooth projective $k$-variety of dimension $d$.  If either i) $d \leq 2$ or ii) $k$ is algebraically closed and $d \leq 3$ , then Question~\ref{quest:Jdescent} admits a positive answer, i.e. every vector bundle on $\tilde{X}$ admits a descent datum relative to $\pi$.
\end{thm}


\subsection{Algebraizability I: obstructions}
\label{subsec:complex}
If $X$ is a smooth complex algebraic variety, then we considered the map
\[
\mathrm{Vect}_r(X) \longrightarrow \mathrm{Vect}_r^{top}(X)
\]
and posed the question of characterizing its image.  We observed that this map factors through the set of motivic vector bundles, so one necessary condition for a topological vector bundle to be algebraizable is that it admits a motivic lift.  In particular, this means that the Chern classes of the topological vector bundle in integral cohomology must lie in the image of the cycle class map.  It is natural to ask if algebraizability of Chern classes is sufficient to guarantee that a vector bundle admits a motivic left.  

In case $X$ is projective, this question has been for instance studied in \cite{Schwarzenberger61} where it is proved that any vector bundle with algebraic Chern classes is algebraizable if $\mathrm{dim}(X)=2$. In case of projective threefolds, positive results are given by Atiyah-Rees and B{\u{a}}nic{\u{a}}-Putinar respectively in \cite{Atiyah76} and \cite{Banica87}. If $X$ is affine, works of Swan-Murthy (\cite{Murthy76}) and Murthy-Mohan Kumar (\cite{Mohan82}) show that the answer to the question is positive if $X$ is of dimension $\leq 3$ as a consequence of the following statement: Given any pair $(\alpha_1,\alpha_2)\in \mathrm{CH}^1(X)\times \mathrm{CH}^2(X)$, there exists a vector bundle $\mathcal{E}$ on $X$ with $c_i(\mathcal{E})=\alpha_i$.  However, in dimension $4$ additional restrictions on Chern classes arise from the action of the motivic Steenrod algebra.

\begin{thm}[{\cite[Theorem 2]{AFHObstructions}}]
If $X$ is a smooth affine $4$-fold, then a pair $(c_1,c_2)\in  \mathrm{CH}^1(X)\times \mathrm{CH}^2(X)$ are Chern classes of a rank $2$ bundle on $X$ if and only if $c_1,c_2$ satisfy the additional condition $\mathrm{Sq}^2(c_2)+c_1 c_2=0$,  where  
\[
\mathrm{Sq}^2:\mathrm{CH}^2(X)\to \mathrm{CH}^3(X)/2
\]
is the Steenrod squaring operation, and $c_1c_2$ is the reduction modulo $2$ of the cup product.	
\end{thm}

\begin{rem}
This obstruction is sufficient to identify topological vector bundles on a smooth affine fourfold $X$ having algebraic Chern classes which are not algebraizable (\cite[Corollary 3.1.5]{AFHObstructions}). One example of such an $X$ is provided by the open complement in $\PP^1\times \PP^3$ of a suitable smooth hypersurface $Z$ of bidegree (3,4)). 
\end{rem}

%

\subsection{Algebraizability II: building motivic vector bundles}
The notion of a cellular space goes back to the work of Dror Farjoun.  By a cellular motivic space, we will mean a space that can be built out of the motivic spheres $S^{p,q}$ by formation of homotopy colimits.  It is straightforward to see inductively that ${\mathbb P}^n$ is cellular.  In the presence of cellularity assumptions, many obstructions to producing a motivic lift of a vector bundle vanish and this motivates the following conjecture.

\begin{conj}
	\label{conj:cellularmotivic}
	If $X$ is a smooth cellular $\cplx$-variety, then the map
	\[
	[X,Gr_r]_{\aone} \longrightarrow \mathrm{Vect}^{top}(X)
	\]
	is surjective (resp. bijective).
\end{conj}

\begin{rem}
	The conjecture holds for ${\mathbb P}^n$ for $n \leq 3$ (this follows, for example, from the results of Schwarzenberger and Atiyah--Rees mentioned above); in these cases, bijectivity holds.  For ${\mathbb P}^4$, the ``surjective'' formulation of Conjecture~\ref{conj:cellularmotivic} is known, but the ``bijective'' formulation is not. 
\end{rem} 	

We now analyze Conjecture~\ref{conj:cellularmotivic} for a class of ``interesting'' topological vector bundles on ${\mathbb P}^n$ introduced by E. Rees and L. Smith.  We briefly recall the construction of these topological vector bundles here.  By a classical result of Serre \cite[Proposition 11]{SerreClasses}, we know that if $p$ is a prime, then the $p$-primary component of $\pi_{4p-3}(S^3)$ is isomorphic to $\Z/p$, generated by the composite of a generator $\alpha_1$ of the $p$-primary component of $\pi_{2p}(S^3)$ and the $(2p-3)$rd suspension of itself; we will write $\alpha_1^2$ for this class.  

The map ${\mathbb P}^n \to S^{2n}$ that collapses ${\mathbb P}^{n-1}$ to a point determines a function
\[
[S^{2n-1},S^3] \cong [S^{2n},BSU(2)] \longrightarrow [{\mathbb P}^n,BSU(2)]
\]
Rees showed that the class $\alpha_1^2$ determines a non-trivial rank $2$ vector bundle $\xi_p \in [{\mathbb P}^{2p-1},BSU(2)]$; we will refer to this bundle as a {\em Rees bundle} \cite{Rees}.  By construction, $\xi_p$ is a non-trivial rank $2$ bundle with trivial Chern classes.

The motivation for Rees' construction originated from results of Grauert--Schneider \cite{GrauertSchneider}.  If the bundles $\xi_p$ were algebraizable, then the fact that they have trivial Chern classes would imply they were necessarily unstable by Barth's results on Chern classes of stable vector bundles \cite[Corollary 1 p. 127]{Barth} (here, stability means slope stability in the sense of Mumford).  Grauert and Schneider analyzed unstable rank $2$ vector bundles on projective space and they aimed to prove that such vector bundles were necessarily direct sums of line bundles; this assertion is now sometimes known as the Grauert--Schneider conjecture.  In view of the Grauert--Schneider conjecture, the bundles $\xi_p$ should not be algebraizable.  On the other hand, one of the motivations for Conjecture~\ref{conj:cellularmotivic} is the following result.  

\begin{thm}[{\cite[Theorem 2.2.16]{Asok20b}}]
	\label{thm:rees}
	For every prime number $p$, the bundle $\xi_p$ lifts to a class in $[\widetilde{{\mathbb P}^{2p-1}},Gr_2]_{\aone}$.
\end{thm}

\begin{rem}
This is established by constructing motivic homotopy classes lifting $\alpha_1$ and $\alpha_1^{2}$.  In our situation, the collapse map takes the form
\[
{\mathbb P}^n \longrightarrow S^{n,n}
\]
and the lift must come from an element of $[S^{n-1,n},SL_2]_{\aone}$.  The class $\alpha_1$ can be lifted using ideas related to those discussed in \ref{subsec:motivicBorelHirzebruch} in conjunction with a motivic version of Serre's classical $p$-local splitting of compact Lie groups \cite[Theorem 2]{AFHLocalization}, the resulting lift has the wrong weight to lift to a group as above.  Since the class $\alpha_1^2$ is torsion, we can employ a weight-shifting mechanism to fix this issue.  In this direction, there are host of other vector bundles that are analogous to the Rees bundles that one might investigate from this point of view, e.g., bundles that can be built out of Toda's unstable $\alpha$-family \cite{Toda}.  Likewise, even the surjectivity assertion in Conjecture~\ref{conj:cellularmotivic} is unknown for ${\mathbb P}^5$.
\end{rem}

\begin{entry}[The Wilson Space Hypothesis]
To close, we briefly sketch an approach to the resolution of Conjecture~\ref{conj:cellularmotivic} building on Mike Hopkins' Wilson Space Hypothesis.  The latter asserts that the Voevodsky motive of the $\pone$-infinite loop spaces $\Omega^{\infty}_{\pone} \Sigma^n_{\pone} \op{MGL}$ arising from algebraic cobordism are pure Tate (the space is ``homologically even''); this hypothesis is an algebro-geometric version of a result of Steve Wilson on the infinite loop spaces of the classical cobordism spectrum.  

The motivic version of the unstable Adams--Novikov resolution for $\op{BGL}_r$ yields a spectral sequence that, under the cellularity assumption on $X$ should converge to a (completion of) the set of rank $r$ motivic vector bundles on $X$.  The resulting spectral sequence can be compared to its topological counterpart and Wilson Space Hypothesis combined with the cellularity assumption on $X$ would imply that the two spectral sequences coincide.   
\end{entry}

\subsubsection*{Acknowledgements}
The authors would like to acknowledge the profound influence of M. Hopkins and F. Morel on their work. Many of the results presented here would not have been possible without their insights and vision. Both authors would also like to thank all their colleagues, coauthors and friends for their help and support. 

\subsubsection*{Funding}
AA was partially supported by NSF Awards DMS-1802060  and DMS-2101898.


{\begin{footnotesize}
\raggedright	
 \bibliographystyle{alpha}
 \bibliography{AsokFasel}
 \end{footnotesize}
}
\Addresses









\end{document}